\documentclass[a4paper]{amsart}
\usepackage{amssymb}
\usepackage[T1]{fontenc}
\usepackage[utf8]{inputenc}
\newtheorem{theorem}{Theorem}[section]
\newtheorem{proposition}[theorem]{Proposition}
\newtheorem{lemma}[theorem]{Lemma}
\newtheorem{corollary}[theorem]{Corollary}

\theoremstyle{remark}
\newtheorem*{remark}{Remark}

\newcommand{\remove}[1]{ }
\def\RR{\mathbb R}
\def\eps{\varepsilon }
\def\gam{\gamma }

\newcommand{\set}[1]{\left\lbrace #1\right\rbrace}
\providecommand{\abs}[1]{\left\lvert#1\right\rvert}
\providecommand{\norm}[1]{\left\lVert#1\right\rVert}

\numberwithin{equation}{section}

\begin{document}

\title[Stabilization of the Gear--Grimshaw system]{Stabilization of the Gear--Grimshaw system on a periodic domain}
\author[Capistrano-Filho]{R. A. Capistrano Filho}
\address{Instituto de Matemática,
Universidade Federal do Rio de Janeiro,
C.P. 68530 - Cidade Universitária - Ilha do Fundão,
21945-970 Rio de Janeiro (RJ),
Brazil}
\email{capistrano@im.ufrj.br}
\author[Komornik]{V. Komornik}
\address{Département de Mathématique,
        Université de Strasbourg,
        7 rue René Descartes,
	67084 Strasbourg Cedex,
	France}
\email{vilmos.komornik@math.unistra.fr}
\author[Pazoto]{A. F. Pazoto}
\address{Instituto de Matemática,
Universidade Federal do Rio de Janeiro,
C.P. 68530 - Cidade Universitária - Ilha do Fundão,
21945-970 Rio de Janeiro (RJ),
Brazil}
\email{ademir@im.ufrj.br}
\date{Version 2013-09-07-a}
\remove{In (1.5) and in the proof of Lemma 3.1 the terms $Mu_x$ and $Nv_x$ were missing where $M=[u]$ and $N=[v]$. The proof was correct in the special case $[\varphi]=[\psi]=0$ but the general case required a small modification.}

\begin{abstract}
This paper is devoted to the study of a nonlinear coupled system of two Korteweg–de Vries equations in a periodic domain
under the effect of an internal damping term. The system was introduced Gear and Grimshaw to model the interactions
of two-dimensional, long, internal gravity waves propagation in a stratified fluid. Designing a time-varying feedback law
and using a Lyapunov approach we establish the exponential stability of the solutions in Sobolev spaces of any positive integral order.
\end{abstract}

\maketitle

\section{Introduction}
\label{s1}

The goal of this paper is to investigate the decay properties of the initial-value problem
\begin{equation}\label{11}
\begin{cases}
u'+uu_x+u_{xxx}+a_3v_{xxx}+a_1vv_x+a_2(uv)_x+k(u-[u])=0,\\
b_1v'+rv_x+vv_x+v_{xxx}+b_2a_3u_{xxx}+b_2a_2uu_x\\
\hspace{5cm} +b_2a_1(uv)_x+k(v-[v])=0,\\
u(0,x)=\phi(x),\\
v(0,x)=\psi(x)
\end{cases}
\end{equation}
with periodic boundary conditions. In \eqref{11}, $r, a_1, a_2, a_3, b_1, b_2, k$ are given real constants with $b_1, b_2, k>0$, $u(t,x), v(t,x)$ are real-valued functions of the time and space variables $t\ge 0$ and $0\le x\le 1$, the subscript $x$ and  the prime indicate the partial differentiation with respect to $x$ and $t$, respectively, and $[f]$ denotes the mean value of $f$ defined by
\begin{equation*}
[f]:=\int_0^1f(x)\ dx.
\end{equation*}
When $k=0$, system was proposed by Gear and Grimshaw \cite{GeaGri1984} as a model to describe strong
interactions of two long internal gravity waves in a stratified fluid, where the two waves
are assumed to correspond to different modes of the linearized equations of motion. It
has the structure of a pair of KdV equations with both linear and nonlinear coupling
terms and has been object of intensive research in recent years. In what concerns the
stabilization problems, most of the works have been focused on a bounded interval with a
localized internal damping (see, for instance, \cite{PazSou} and the references therein). In particular,
we also refer to \cite{BonPonSauTom1992} for an extensive discussion on the physical relevance
of the system and to \cite{Dav1,Dav4,Dav2,Dav3,DavCha2006} for the results used in this
paper.

We can (formally) check that the total energy 
\begin{equation*}
 E=\frac{1}{2}\int_0^1  b_2u^2 +  b_1v^2\ dx
\end{equation*}
associated with the model satisfies the inequality
\begin{equation*}
E' = -k\int_0^1  b_2(u-[u])^2 +  (v-[v])^2\ dx\le 0
\end{equation*}
in $(0,\infty)$, so that the energy in nonincreasing.
Therefore, the following basic questions arise: are the solutions asymptotically stable for
t sufficiently large? And if yes, is it possible to find a rate of decay? The aim of this paper is to answer these questions.

More precisely, we prove that for any fixed  integer $s\ge 3$, the solutions are exponentially stable in the Sobolev spaces
\begin{equation*}
H^{s}_{p}(0,1):=\set{u\in H^s(0,1)\ :\ \partial^n_x u(0)=\partial_x^n u(1),\quad n=0,\ldots,s}
\end{equation*}
with periodic boundary conditions. This extends an earlier theorem of Dávila in \cite{Dav3} for $s\le 2$. 

Before stating the stabilization result mentioned above, we first need to ensure the well posedness of the system.
This  was addressed by Dávila in \cite{Dav1} (see also \cite{Dav4}) under the following conditions on the coefficients:
\begin{equation}\label{12}
\begin{matrix}
a_3^2 b_2 < 1\text{ and } r=0\\
b_2a_1a_3 -b_1a_3 +b_1a_2 - a_2=0\\
b_1a_1-a_1-b_1a_2a_3+a_3=0\\
b_1a_2^2+b_2a_1^2-b_1a_1-a_2=0.
\end{matrix}
\end{equation}
Indeed, under conditions \eqref{12}, Dávila and Chaves \cite{DavCha2006} derived some conservation laws
for the solutions of  \eqref{11}. Combined with an approach introduced in \cite{BonSmi,SauTem}, these conservation laws allow them to establish the global well-posedness in $H^s_p(0,1)$, for any $s\geq 0$.  Moreover, the authors also give a simpler derivation of  the conservation laws discovered by Gear and Grimshaw, and Bona et al \cite{BonPonSauTom1992}. We also observe that these conservation properties were obtained
employing the techniques developed in \cite{MiuGarKru1968}  for the single KdV equation;
see also \cite{Miu1976}.

The well-posedness result reads as follows:

\begin{theorem}\label{t11}
Assume that condition \eqref{12} holds. If $\phi,\psi\in H^{s}_{p}(0,1)$ for some integer $s\ge 3$, then the system \eqref{11} has a unique solution satisfying
\begin{equation*}
u,v\in C([0,\infty);H^{s}_{p}(0,1))\cap C^1([0,\infty);H^{s-3}_{p}(0,1)).
\end{equation*}
Moreover, the map $(\phi,\psi)\mapsto (u,v)$ is continuous from $\left( H^{s}_{p}(0,1)\right) ^2$ into
\begin{equation*}
\left( C([0,\infty);H^{s}_{p}(0,1))\cap C^1([0,\infty);H^{s-3}_{p}(0,1))\right) ^2.
\end{equation*}
\end{theorem}

For $k=0$, the analogous theorem on the whole real line $-\infty<x<\infty$
was proved Bona et al. \cite{BonPonSauTom1992}, for all $s\ge 1$.

With the global well-posedness result in hand, we can focus on the stabilization problem.
For simplicity of notation we consider only the case
\begin{equation}\label{13}
b_1=b_2=1.
\end{equation}
Then the conditions \eqref{12} take the simplified form 
\begin{equation}\label{14}
r=0,\quad a_1^2+a_2^2=a_1+a_2,\quad \abs{a_3}<1,\quad\text{and}\quad  (a_1-1)a_3=(a_2-1)a_3=0.
\end{equation}
Hence either $a_3=0$ and $a_1^2+a_2^2=a_1+a_2$, or  $0<\abs{a_3}<1$ and $a_1=a_2=1$.

We prove the following theorem:

\begin{theorem}\label{t12}
Assume  \eqref{13} and \eqref{14}.
If $\phi,\psi\in H^{s}_{p}(0,1)$ for some integer $s\ge 3$, then the solution of \eqref{11} satisfies the estimate
\begin{equation*}
\norm{u(t)-[u(t)]}_{H^{s}_{p}(0,1)}+\norm{v(t)-[v(t)]}_{H^{s}_{p}(0,1)}=o\left( e^{-k't}\right) ,\quad t\to\infty
\end{equation*}
for each $k'<k$.
\end{theorem}

An analogous theorem was proved in \cite{KomRusZha1991} for the usual KdV equation by using the infinite family
of conservation laws for this equation. Such conservations lead to the construction of a suitable Lyapunov function that gives the exponential decay of the solutions. Here, we
follow the same approach making use of the results established by Dávila and Chavez \cite{DavCha2006}.
They proved that under the assumptions \eqref{12} system \eqref{11}
also has an infinite family of conservation laws, and they conjectured the above theorem for this case.

In order to obtain the result, we prove a number of identities and estimates for the solutions of  \eqref{11}. In view of Theorem \ref{t11} it suffices to establish these estimates for \emph{smooth solutions}, i.e., to solutions corresponding to $C^{\infty}$ initial data $\phi,\psi$ with periodic boundary conditions. For such solutions all formal manipulations in the sequel will be justified.

Finally, we also observe that a similar result was obtained in \cite{LauRosZha2010} for the scalar KdV equation on a periodic domain. The authors study the model from a control point of view with a forcing term $f$ supported in a given open set of the domain. It is shown that the system is globally exactly controllable and globally exponentially stable. The stabilization is established with the aid of certain properties of propagation of compactness and regularity in Bourgain spaces for the solutions of the corresponding linear system. We also refer to \cite{LauRosZha2010} for a quite complete review on the subject.

The paper is organized as follows. In Section 2 introduce the basic notations and we prove some technical lemmas. Sections 3  to 6 are devoted to the proof of the exponential decay in $H^s_p$, for $s=0, 1, 2$ and $s\geq 3$, respectively. 

\section{Some technical lemmas}
\label{s2}

In the sequel all integrals are taken over the interval $(0,1)$ so we omit the integration limits.

As explained in the introduction, all integrations by parts will be done for smooth periodic functions. Therefore, we will regularly use the simplified formulas 
\begin{equation*}
\int f_xg\ dx=-\int fg_x\ dx
\quad\text{and}\quad
\int f^nf_x\ dx=0\quad(n=0,1,\ldots)
\end{equation*}
without further explanation, and we will also use the simplified notation 
\begin{equation*}
f_n:=\frac{d^nf}{dx^n},\quad n=1,2,\ldots
\end{equation*}
As an example of the application of these rules we show that the mean-values of of the solutions are conserved:

\begin{lemma}\label{l21}
The mean-values $[u]$ and $[v]$ of the solutions of \eqref{11} do not depend on $t$.
\end{lemma}

\begin{proof}
We have
\begin{align*}
[u]'
&=-\int uu_x+u_{xxx}+a_3v_{xxx}+a_1vv_x+a_2(uv)_x+k(u-[u])\ dx\\
&=-\int \left( \frac{u^2}{2}+u_{xx}+a_3v_{xx}+a_1\frac{v^2}{2}+a_2uv\right)_x+k(u-[u])\ dx\\
&=-k\int (u-[u])\ dx\\
&=0
\end{align*}
and
\begin{align*}
[v]'
&=-\int vv_x+v_{xxx}+a_3u_{xxx}+a_2uu_x+a_1(uv)_x+k(v-[v])\ dx\\
&=-\int \left( \frac{v^2}{2}+v_{xx}+a_3u_{xx}+a_2\frac{u^2}{2}+a_1uv\right)_x+k(v-[v])\ dx\\
&=-k\int (v-[v])\ dx\\
&=0
\end{align*}
by a straightforward computation.
\end{proof}

Motivated by this result we set $M=[\varphi]$, $N=[\psi]$ and we rewrite \eqref{11} by changing $u, v, \varphi, \psi$ to $u-[u]=u-M$, $v-[v]=v-N$, $\varphi-[\varphi]=\varphi-M$ and $\psi-[\psi]=\psi-N$, respectively. Under our assumptions $r=0$ and $b_1=b_2=1$ we obtain the equivalent system
\begin{equation}\label{21}
\begin{cases}
u'+(u+M)u_x+u_{xxx}+a_3v_{xxx}+a_1(v+N)v_x\\
\hspace*{5cm}+a_2((u+M)(v+N))_x+ku=0,\\
v'+(v+N)v_x+v_{xxx}+a_3u_{xxx}+a_2(u+M)u_x\\
\hspace*{5cm}+a_1((u+M)(v+N))_x+kv=0,\\
u(0,x)=\phi(x),\\
v(0,x)=\psi(x)
\end{cases}
\end{equation}
with periodic boundary conditions, corresponding to initial data $\phi,\psi$ with zero mean values.
Theorem \ref{t12} will thus follow from the following proposition:

\begin{proposition}\label{p22}
Under the assumptions of Theorem \ref{t12} the smooth solutions of \eqref{21} satisfy the identity
\begin{equation}\label{22}
\int u(t)^2+v(t)^2\ dx=e^{-2kt}\int \phi^2+\psi^2\ dx,\quad t\ge 0,
\end{equation}
and the estimates
\begin{equation*}
e^{2k't}\int \left( \partial_x^nu(t)\right) ^2+\left( \partial_x^nv(t)\right) ^2 dx\to 0\quad\text{as}\quad t\to\infty
\end{equation*}
for all positive integers $n$ and for all  $k'<k$.
\end{proposition}

\begin{remark}
For $n=1$ the proposition and its proof remain valid under the weaker assumption that $\abs{a_3}<1$. We can also add the term $rv_x$ to the equation by changing $g$ to $g-rv^2$ in Lemma \ref{l41}.
\end{remark}

Proposition \ref{p22} is proved by using the Lyapunov method. More precisely, we shall use the following lemma:

\begin{lemma}\label{l23}
Let $f:(0,\infty)\to\RR$ be a nonnegative function, and write $h_1\approx h_2$ if $h_1-h_2=o(f)$ as $t\to\infty$.

If there exists a function $g:(0,\infty)\to\RR$ such that $g\approx 0$, $f+g$ is continuously differentiable, and
$(f+g)'\approx -2kf$ for some positive number $k$, then
\begin{equation*}
e^{2k't}f(t)\to 0\quad\text{as}\quad t\to\infty
\end{equation*}
for each $k'<k$.
\end{lemma}

\begin{proof}
Fix $k''>0$ such that $k'<k''<k$, and then fix  $\eps>0$ such that
\begin{equation*}
\frac{1-\eps}{1+\eps}=\frac{k''}{k}.
\end{equation*}
Finally, choose a sufficiently large $t'>0$ such that
\begin{equation*}
(1-\eps)f(t)\le (f+g)(t)\le (1+\eps)f(t)
\end{equation*}
and
\begin{equation*}
2k(1-\eps)f(t)\le -(f+g)'(t)\le 2k(1+\eps)f(t)
\end{equation*}
for all $t\ge t'$. Then for $t\ge t'$ we have
\begin{equation*}
-(f+g)'(t)\ge 2k(1-\eps)f(t)\ge 2k\frac{1-\eps}{1+\eps}(f+g)(t)=2k''(f+g)(t),
\end{equation*}
whence
\begin{equation*}
\frac{d}{dt}\left( e^{2k''t}(f+g)(t)\right)\le 0.
\end{equation*}
It follows that
\begin{equation*}
e^{2k''t}(f+g)(t)\le  e^{2k''t'}(f+g)(t')
\end{equation*}
for all $t\ge t'$, and hence
\begin{equation*}
0\le e^{2k't}f(t)\le  \frac{e^{2k''t'}(f+g)(t')}{1-\eps} e^{-2(k''-k')t}
\end{equation*}
for all $t\ge t'$. We conclude by observing that $e^{-2(k''-k')t}\to 0$ as $t\to\infty$.
\end{proof}

For the proof of the next result, we shall use the H\"older and Poincaré--Wirtinger inequalities in the following form.
The second estimate will be used only for functions with mean value zero: $[u]=0$.

\begin{lemma}\label{l24} 
If $p, q \in [0,\infty)$, then
\begin{align}
&\norm{u}_p\le \norm{u}_q\quad\text{for all}\quad u\in L^q(0,1)\quad\text{and}\quad 1\le p\le q\le\infty;\label{23}\\
&\norm{u-[u]}_p\le \norm{u_x}_q\quad\text{for all}\quad u\in H^1(0,1)\quad\text{and}\quad 1\le p, q\le\infty.\label{24}
\end{align}
\end{lemma}

We shall frequently use Lemma \ref{l23} together with the following result:

\begin{lemma}\label{l25}
Let $n\ge 1$ and let $\alpha_m, \beta_m$, $m=0,\ldots,n$ be nonnegative integers satisfying the two conditions
\begin{equation*}
2(\alpha_n+\beta_n)+\alpha_{n-1}+\beta_{n-1}\le 4
\end{equation*}
and
\begin{equation*}
d:=\sum_{m=0}^n\left( \alpha_m+\beta_m\right) \ge 2.
\end{equation*}
Then
\begin{equation*}
\abs{\int\prod_{m=0}^nu_m^{\alpha_m}v_m^{\beta_m}\ dx}\le \left( \int u_n^2+v_n^2 \ dx\right) \left( \int u_{n-1}^2+v_{n-1}^2 \ dx\right) ^{\frac{d-2}{2}}.
\end{equation*}

If, moreover, $d\ge 3$ and
\begin{equation*}
\int u_{n-1}^2+v_{n-1}^2 \ dx\to 0,
\end{equation*}
then it follows that
\begin{equation*}
\int\prod_{m=0}^nu_m^{\alpha_m}v_m^{\beta_m}\ dx=o\left( \int u_n^2+v_n^2 \ dx\right)
\end{equation*}
as $t\to\infty$.
\end{lemma}

\begin{proof}
Setting
\begin{equation*}
z_m:=\sqrt{u_m^2+v_m^2}\quad \text{and}\quad \gam_m:=\alpha_m+\beta_m,\quad m=0,\ldots, n
\end{equation*}
we have
\begin{equation*}
\abs{\int\prod_{m=0}^nu_m^{\alpha_m}v_m^{\beta_m}\ dx}\le \int\prod_{m=0}^nz_m^{\gam_m}\ dx.
\end{equation*}
We are going to majorize the right side by using the H\"older and Poincaré--Wirtinger inequalities \eqref{23}--\eqref{24}. We distinguish five cases according to the value of $\gam_n+\gam_{n-1}$: since $2\gam_n+\gam_{n-1}\le 4$ by our assumption, $\gam_n+\gam_{n-1}\le 4$.

If $\gam_n+\gam_{n-1}=0$, then we have
\begin{equation*}
\abs{\int\prod_{m=0}^nz_m^{\gam_m}\ dx}
\le \prod_{m=0}^{n-2}\norm{z_m}_{\infty}^{\gam_m}
\le \norm{z_n}_2^2\norm{z_{n-1}}_2^{d-2}.
\end{equation*}

If $\gam_n+\gam_{n-1}=1$, then
\begin{equation*}
\abs{\int\prod_{m=0}^nz_m^{\gam_m}\ dx}
\le \norm{z_n}_1 \prod_{m=0}^{n-2}\norm{z_m}_{\infty}^{\gam_m}
\le \norm{z_n}_2^2\norm{z_{n-1}}_2^{d-2}.
\end{equation*}

If $\gam_n+\gam_{n-1}=2$, then
\begin{equation*}
\abs{\int\prod_{m=0}^nz_m^{\gam_m}\ dx}
\le \norm{z_n}_2^2\prod_{m=0}^{n-2}\norm{z_m}_{\infty}^{\gam_m}
\le \norm{z_n}_2^2\norm{z_{n-1}}_2^{d-2}.
\end{equation*}

If $\gam_n+\gam_{n-1}=3$, then we have necessarily $\gam_n=1$ and $\gam_{n-1}=2$, so that
\begin{equation*}
\abs{\int\prod_{m=0}^nz_m^{\gam_m}\ dx}
\le \norm{z_n}_2\norm{z_{n-1}}_{\infty}\norm{z_{n-1}}_2\prod_{m=0}^{n-2}\norm{z_m}_{\infty}^{\gam_m}
\le \norm{z_n}_2^2\norm{z_{n-1}}_2^{d-2}.
\end{equation*}

Finally, if $\gam_n+\gam_{n-1}=4$, then we have necessarily $\gam_n=0$ and $\gam_{n-1}=4$, so that
\begin{equation*}
\abs{\int\prod_{m=0}^nz_m^{\gam_m}\ dx}
\le \norm{z_{n-1}}_{\infty}^2\norm{z_{n-1}}_2^2\prod_{m=0}^{n-2}\norm{z_m}_{\infty}^{\gam_m}
\le \norm{z_n}_2^2\norm{z_{n-1}}_2^{d-2}.\qedhere
\end{equation*}
\end{proof}

\section{Proof of Proposition \ref{p22} for $n=0$}\label{s3}


Our proof is based on the following identity:

\begin{lemma}\label{l31}
The solutions of \eqref{21} satisfy the following identity for all $n=0,1,\ldots :$
\begin{align}\label{31}
\left( \int u_n^2+v_n^2\ dx\right) '
&=-2k\int u_n^2+v_n^2\ dx\\
&\qquad\qquad -2\int u_n(u_1u)_n+v_n(v_1v)_n\ dx\notag \\
&\qquad\qquad -2a_1\int u_n(vv_1)_n+v_n(uv)_{n+1}\ dx\notag \\
&\qquad\qquad -2a_2\int v_n(uu_1)_n+u_n(uv)_{n+1}\ dx.\notag
\end{align}
\end{lemma}

\begin{proof}
We have
\begin{align*}
\left( \int u_n^2+v_n^2\ dx\right) '
&=\int 2u_nu_n'+2v_nv_n'\ dx\\
&=\int -2u_n((u+M)u_1+u_3+a_3v_3+a_1(v+N)v_1\\
&\hspace*{4cm}+a_2((u+M)(v+N))_1+ku)_n\ dx\\
&\qquad +\int -2v_n((v+N)v_1+v_3+a_3u_3+a_2(u+M)u_1\\
&\hspace*{4cm} +a_1((u+M)(v+N))_1+kv)_n\ dx.
\end{align*}
This yields the stated identity because
\begin{align*}
\int -2u_nu_{n+3}-2v_nv_{n+3}\ dx&=\int 2u_{n+1}u_{n+2}+2v_{n+1}v_{n+2}\ dx\\
&=\int (u_{n+1}^2)_1+(v_{n+1}^2)_1\ dx=0,
\end{align*}
\begin{equation*}
a_3\int -2u_nv_{n+3}-2v_nu_{n+3}\ dx=a_3\int -2u_nv_{n+3}+2v_{n+3}u_n\ dx=0,
\end{equation*}
\begin{multline*}
-2M\int u_nu_{n+1}+a_2u_nv_{n+1}+a_2v_nu_{n+1}+a_1v_nv_{n+1}\ dx\\
=-M\int\left( u_n^2+2a_2u_nv_n+a_1v_n^2\right) _1\ dx=0,
\end{multline*}
\begin{multline*}
-2N\int a_1u_nv_{n+1}+a_2u_nu_{n+1}+v_nv_{n+1}+a_1v_nu_{n+1}\ dx\\
=-N\int\left( 2a_1u_nv_n+a_2u_n^2+v_n^2\right) _1\ dx
=0
\end{multline*}
and $(MN)_1=0$.
\end{proof}

\begin{proof}[Proof of the proposition for $n=0$]
In this case the last three integrals of the identity \eqref{31} vanish because
\begin{align*}
&\int uu_1u+vv_1v\ dx=\frac{1}{3}\int (u^3+v^3)_1\ dx=0,\\
&\int uvv_1+v(uv)_1\ dx=\int (uvv)_1\ dx=0
\intertext{and}
&\int vuu_1+u(uv)_1\ dx=\int (vuu)_1\ dx=0.\qedhere
\end{align*}
\end{proof}

Proceeding by induction on $n$, let $n\ge 1$ and assume that the estimates
\begin{equation}\label{32}
\int u_m^2+v_m^2 \ dx=o\left( e^{-2k't}\right)\quad\text{as}\quad t\to\infty
\end{equation}
hold for all integers $m=0,\ldots,n-1$ and for all  $k'<k$. For $n=1$ this follows from the stronger identity \eqref{22}.

\section{Proof of Proposition \ref{p22} for $n=1$}\label{s4}

For the proof of the case $n=1$ we shall use an identity suggested by a conservation law discovered by Bona et al. \cite{BonPonSauTom1992}.

\begin{lemma}\label{l41}
Setting
\begin{equation*}
f:=\int u_1^2+v_1^2+2a_3u_1v_1\ dx
\end{equation*}
and
\begin{equation*}
g:=-\frac{1}{3}\int (u^3+v^3)+3(a_1uv^2+a_2u^2v)\ dx,
\end{equation*}
we have the following identity:
\begin{equation}\label{41}
(f+g)'=-2kf-3kg.
\end{equation}
\end{lemma}

\begin{proof}
The equality \eqref{41} will follow by combining the following four identities:

\begin{align}\label{42}
\left( \int u_1^2+v_1^2\ dx\right) '
&=-2k\int u_1^2+v_1^2\ dx\\
&\qquad\qquad -\int u_1^3+v_1^3\ dx\notag \\
&\qquad\qquad  -3a_1\int u_1v_1^2\ dx\notag\\
&\qquad\qquad -3a_2\int u_1^2v_1\ dx;\notag
\end{align}

\begin{align}\label{43}
\left( \int u_1v_1\ dx\right) '
&=-2k\int u_1v_1\ dx+\int uu_1v_2+vv_1u_2\ dx\\
&\qquad\qquad -\frac{a_1}{2}\int 2v_2u_1u+3v_1u_1^2+v_1^3\ dx\notag \\
&\qquad\qquad-\frac{a_2}{2}\int 2u_2v_1v+3u_1v_1^2+u_1^3\ dx;\notag
\end{align}

\begin{align}\label{44}
\left( \int u^3+v^3\ dx\right) '
&=-3k\int u^3+v^3\ dx-3\int u_1^3+v_1^3\ dx\\
&\qquad\qquad -a_1\int 3u^2vv_1+2v^3u_1\ dx\notag \\
&\qquad\qquad -a_2\int 3v^2uu_1+2 u^3v_1\ dx.\notag \\
&\qquad\qquad +6a_3\int uu_1v_2+vv_1u_2\ dx;\notag
\end{align}

\begin{align}\label{45}
\left( \int a_1uv^2+a_2u^2v\ dx\right) '
&=-3k \int a_1uv^2+a_2u^2v\ dx\\
&\qquad\qquad +a_1\int \frac{2}{3}v^3u_1+u^2vv_1-3v_1^2u_1\ dx\notag\\
&\qquad\qquad +a_2\int \frac{2}{3}u^3v_1+v^2uu_1-3u_1^2v_1\ dx\notag\\
&\qquad\qquad -a_1a_3\int 2v_2u_1u+3v_1u_1^2+v_1^3\ dx\notag\\
&\qquad\qquad -a_2a_3\int 2u_2v_1v+3u_1v_1^2+u_1^3\ dx.\notag
\end{align}
\medskip

\emph{Proof of \eqref{42}.} We transform the identity \eqref{31} for $n=1$ as follows.
We have
\begin{align*}
\int u_1(u_1u)_1+v_1(v_1v)_1\ dx
&=\int u_2u_1u+u_1^3+v_2v_1v+v_1^3\ dx\\
&=\int u_1^3+v_1^3+\frac{1}{2}(u_1^2)_1u+\frac{1}{2}(v_1^2)_1v\ dx\\
&=\frac{1}{2}\int u_1^3+v_1^3\ dx,
\end{align*}
\begin{align*}
\int u_1(vv_1)_1+v_1(uv)_2\ dx
&=\int u_1v_1^2+u_1vv_2-v_2(uv)_1\ dx\\
&=\int u_1v_1^2-v_2uv_1\ dx\\
&=\int u_1v_1^2-\frac{1}{2}u(v_1^2)_1\ dx\\
&=\frac{3}{2}\int u_1v_1^2\ dx,
\end{align*}
and by symmetry
\begin{equation*}
\int v_1(uu_1)_1+u_1(uv)_2\ dx=\frac{3}{2}\int u_1^2v_1\ dx.
\end{equation*}
Using them \eqref{31} implies \eqref{42}.
\medskip

\emph{Proof of \eqref{43}.} We have
\begin{align*}
\left( \int u_1v_1\ dx\right) '
&=\int u_1'v_1+u_1v_1'\ dx\\
&=\int -(uu_1+u_3+a_3v_3+a_1vv_1+a_2(uv)_1+ku)_1v_1\ dx\\
&\qquad\qquad +\int -u_1(vv_1+v_3+a_3u_3+a_2uu_1+a_1(uv)_1+kv)_1\ dx\\
&=-2k\int u_1v_1\ dx+\int (uu_1+u_3)v_2+(vv_1+v_3)u_2\ dx\\
&\qquad\qquad -a_1\int (vv_1)_1v_1+u_1(uv)_2\ dx\\
&\qquad\qquad -a_2\int (uv)_2v_1+u_1(uu_1)_1\ dx\\
&\qquad\qquad -a_3\int v_4v_1+u_4u_1\ dx\\
&=-2k\int u_1v_1\ dx+\int uu_1v_2+vv_1u_2\ dx\\
&\qquad\qquad +a_1\int vv_1v_2+u_2(uv)_1\ dx\\
&\qquad\qquad +a_2\int (uv)_1v_2+u_2uu_1\ dx
\end{align*}
because
\begin{equation*}
\int u_3v_2+v_3u_2\ dx=\int u_3v_2-v_2u_3\ dx=0
\end{equation*}
and
\begin{equation*}
\int v_4v_1+u_4u_1\ dx=-\int v_3v_2+u_3u_2\ dx=-\frac{1}{2}\int (v_2^2+u_2^2)_1\ dx=0.
\end{equation*}

Since
\begin{align*}
\int vv_1v_2+u_2(uv)_1\ dx&=\int \frac{1}{2}v(v_1^2)_1+\frac{1}{2}(u_1^2)_1v+u_2uv_1\ dx\\
& =\int -\frac{1}{2}v_1^3-\frac{1}{2}u_1^2v_1-u_1^2v_1-u_1uv_2\ dx\\
& =-\frac{1}{2}\int 2v_2u_1u+3v_1u_1^2+v_1^3\ dx
\end{align*}
and by symmetry
\begin{equation*}
\int uu_1u_2+v_2(uv)_1\ dx=-\frac{1}{2}\int 2u_2v_1v+3u_1v_1^2+u_1^3\ dx,
\end{equation*}
\eqref{43} follows from the previous identity.
\medskip

\emph{Proof of \eqref{44}.} We have
\begin{align*}
\left( \int u^3\ dx\right) '
&=\int 3u^2u'\ dx\\
&=\int -3u^2(uu_1+u_3+a_3v_3+a_1vv_1+a_2(uv)_1+ku)\ dx\\
&= \int -\frac{3}{4}\left( u^4\right) _1+3u\left( u_1^2\right) _1-3ku^3\ dx
-3a_3\int u^2v_3\ dx\\
&\qquad\qquad -3a_1\int u^2vv_1\ dx-3a_2\int u^3v_1+\frac{1}{3}(u^3)_1v\ dx\\
&=-3\int u_1^3+ku^3\ dx-3a_1\int u^2vv_1\ dx-2a_2\int u^3v_1\ dx\\
&\qquad\qquad +6a_3\int uu_1v_2\ dx.
\end{align*}
We have an analogous identity for $\int v^3\ dx$ by symmetry; adding the we get \eqref{44}.
\medskip

\emph{Proof of \eqref{45}.} We have
\begin{align*}
\left( \int u^2v\ dx\right) '
&=\int u'(2uv)+u^2v'\ dx\\
&=\int -2uv(uu_1+u_3+a_3v_3+a_1vv_1+a_2(uv)_1+ku)\ dx\\
&\qquad\qquad +\int -u^2(vv_1+v_3+a_3u_3+a_2uu_1+a_1(uv)_1+kv)\ dx\\
&=\int -2u^2u_1v+2u_2(uv)_1-u^2vv_1+2v_2uu_1\ dx-3k\int u^2v\ dx\\
&\qquad\qquad -a_1\int 2uvvv_1+u^2(uv)_1\ dx\\
&\qquad\qquad -a_2\int 2uv(uv)_1+u^3u_1\ dx\\
&\qquad\qquad -a_3\int 2uvv_3+u^2u_3\ dx.
\end{align*}
Here
\begin{equation*}
\int -2u^2u_1v\ dx=-\frac{2}{3}\int (u^3)_1v\ dx=\frac{2}{3}\int u^3v_1,
\end{equation*}
\begin{equation*}
\int -u^2vv_1\ dx=-\frac{1}{2}\int u^2(v^2)_1\ dx=\frac{1}{2}\int (u^2)_1v^2\ dx=\int v^2uu_1\ dx,
\end{equation*}
\begin{align*}
\int 2u_2(uv)_1+2v_2uu_1\ dx
&=\int (2u_2u_1v+2u_2uv_1)-(2v_1u_1^2+2v_1uu_2)\ dx\\
&=\int (u_1^2)_1v-2v_1u_1^2\ dx\\
&=-3\int u_1^2v_1\ dx,
\end{align*}
\begin{equation*}
\int 2uvvv_1+u^2(uv)_1\ dx=\int \frac{2}{3}u(v^3)_1+u^3v_1+\frac{1}{3}(u^3)_1v\ dx=\frac{2}{3}\int u^3v_1-v^3u_1\ dx,
\end{equation*}
\begin{equation*}
\int 2uv(uv)_1+u^3u_1\ dx=\int \left( (uv)^2+\frac{1}{4}u^4\right) _1\ dx=0,
\end{equation*}
and
\begin{align*}
\int 2uvv_3+u^2u_3\ dx
&=\int -2(u_1v+uv_1)v_2-2uu_1u_2\ dx\\
&=\int 2(u_2v+u_1v_1)v_1-u(v_1^2)_1-u(u_1^2)_1\ dx\\
&=\int 2(u_2v+u_1v_1)v_1+u_1v_1^2+u_1^3\ dx\\
&=\int 2u_2v_1v+3u_1v_1^2+u_1^3\ dx,
\end{align*}
so that
\begin{multline*}
\left( \int u^2v\ dx\right) '=\int \frac{2}{3}u^3v_1+v^2uu_1-3u_1^2v_1\ dx-3k\int u^2v\ dx\\
- \frac{2}{3}a_1\int u^3v_1-v^3u_1\ dx-a_3\int 2u_2v_1v+3u_1v_1^2+u_1^3\ dx.
\end{multline*}

By symmetry, we also have
\begin{multline*}
\left( \int v^2u\ dx\right) '=\int \frac{2}{3}v^3u_1+u^2vv_1-3v_1^2u_1\ dx-3k\int v^2u\ dx\\
- \frac{2}{3}a_2\int v^3u_1-u^3v_1\ dx-a_3\int 2v_2u_1u+3v_1u_1^2+v_1^3\ dx.
\end{multline*}

Combining the last two identities  \eqref{45}  follows (some terms annihilate each other).
\end{proof}

\begin{proof}[Proof of the proposition for $n=1$]
It suffices to show that the functions $f$ and $g$ of Lemma \ref{l41} satisfy the conditions of Lemma \ref{l23}.
Since $\abs{a_3}<1$, we have $f\ge 0$. The other conditions follow from the already proven case $n=0$ and from the second part of Lemma \ref{l25}. We conclude by applying the lemma and then by observing that
\begin{equation*}
\int u_1^2+v_1^2\ dx\le \frac{1}{1-\abs{a_3}}\int u_1^2+v_1^2+2a_3u_1v_1\ dx.\qedhere
\end{equation*}
\end{proof}

\section{Proof of Proposition \ref{p22} for $n=2$}\label{s5}

\begin{lemma}\label{l51}
Setting
\begin{align*}
&f:=\int u_2^2+v_2^2+2a_3u_2v_2\ dx,\\
&g:=-\frac{5}{3}\int (u_1^2u+v_1^2v)+a_1(2u_1v_1v+v_1^2u)+a_2(2u_1v_1u+u_1^2v)\ dx
\intertext{and}
&h:=\frac{2}{3}a_3\int (1-a_1)(2u_3v_2u+u_2v_2u_1)+(1-a_2)(2v_3u_2v+u_2v_2v_1)\ dx,
\end{align*}
we have
\begin{equation}\label{51}
(f+g)'\approx -2kf+h.
\end{equation}
\end{lemma}

\begin{proof}
The relationship \eqref{51} will follow by combining the following relations:

\begin{align}\label{52}
\left( \int u_2^2+v_2^2\ dx\right) '
&=-2k\int u_2^2+v_2^2\ dx \\
&\qquad\qquad -5\int u_2^2u_1+v_2^2v_1\ dx\notag  \\
&\qquad\qquad -5a_1\int 2u_2v_2v_1+v_2^2u_1\ dx \notag \\
&\qquad\qquad -5a_2\int 2u_2v_2u_1+u_2^2v_1\ dx;\notag
\end{align}

\begin{align}\label{53}
&\left( \int u_2v_2\ dx\right)'= -2k\int u_2v_2\ dx\\
&\qquad\qquad -\int u_3v_2u+v_3u_2v+3u_2v_2(u_1+v_1)\ dx\notag\\
&\qquad\qquad -a_1\int \frac{5}{2}(u_2^2+v_2^2)v_1+2u_2v_2u_1-u_3v_2u \ dx\notag\\
&\qquad\qquad -a_2\int \frac{5}{2}(u_2^2+v_2^2)u_1+2u_2v_2v_1-v_3u_2v\ dx;\notag
\end{align}

\begin{align}\label{54}
\left( \int u_1^2u+v_1^2v\ dx\right) '
&\approx -3\int u_2^2u_1+v_2^2v_1\ dx \\
&\qquad\qquad -2a_3\int u_3v_2u+v_3u_2v+2u_2v_2(u_1+v_1)\ dx;\notag
\end{align}

\begin{align}\label{55}
\left( \int 2u_1v_1v+v_1^2u\ dx\right) '
&\approx -3\int 2u_2v_2v_1+v_2^2u_1\ dx \\
&\qquad +a_3\int -3(u_2^2+v_2^2)v_1+2u_3v_2u-2u_2v_2u_1\ dx;\notag
\end{align}

\begin{align}\label{56}
\left( \int 2u_1v_1u+u_1^2v\ dx\right) '
&\approx -3\int 2u_2v_2u_1+u_2^2v_1\ dx \\
&\qquad +a_3\int -3(u_2^2+v_2^2)u_1+2v_3u_2v-2u_2v_2v_1\ dx.\notag
\end{align}

\emph{Proof of \eqref{52}.} We transform the last three integrals of the identity \eqref{31} in the following way:
\begin{align*}
-2\int u_2(u_1u)_2+v_2(v_1v)_2\ dx
&=-2\int 3u_2^2u_1+u_2u_3u+3v_2^2v_1+v_2v_3v\ dx\\
&=-2\int 3u_2^2u_1+\frac{1}{2}(u_2^2)_1u+3v_2^2v_1+\frac{1}{2}(v_2^2)_1v\ dx\\
&=-5\int u_2^2u_1+v_2^2v_1\ dx,
\end{align*}
\begin{align*}
-2a_1\int u_2(vv_1)_2+v_2(uv)_3\ dx
&=-2a_1\int 3u_2v_1v_2+u_2vv_3-v_3(uv)_2\ dx\\
&=-2a_1\int 3u_2v_1v_2-2v_3u_1v_1-v_3uv_2\ dx\\
&=-2a_1\int 3u_2v_1v_2+2v_2(u_1v_1)_1-\frac{1}{2}u(v_2^2)_1\ dx\\
&=-2a_1\int 5u_2v_1v_2+\frac{5}{2}u_1v_2^2\ dx\\
&=-5a_1\int 2u_2v_2v_1+v_2^2u_1\ dx,
\end{align*}
and by symmetry
\begin{equation*}
-2a_2\int v_2(uu_1)_2+u_2(uv)_3\ dx=-5a_2\int 2u_2v_2u_1+u_2^2v_1\ dx.
\end{equation*}

Combining these identities with \eqref{31} we obtain \eqref{52}.
\medskip

\emph{Proof of \eqref{53}.} We have

\begin{align*}
\left( \int u_2v_2\ dx\right)'
&= \int u_2'v_2+u_2v_2'\ dx\\
&=-\int (u_1u+u_3+ku+a_3v_3+a_1v_1v+a_2(uv)_1)_2v_2\ dx\\
&\qquad\qquad -\int u_2(v_1v+v_3+kv+a_3u_3+a_2u_1u+a_1(uv)_1)_2 \ dx\\
&=-2k\int u_2v_2\ dx\\
&\qquad\qquad-a_3\int v_5v_2+u_2u_5\ dx-\int u_5v_2+u_2v_5\ dx\\
&\qquad\qquad -\int (uu_1)_2v_2+u_2(vv_1)_2\ dx\\
&\qquad\qquad -a_1\int (vv_1)_2v_2+u_2(uv)_3\ dx\\
&\qquad\qquad -a_2\int (uv)_3v_2+u_2(uu_1)_2\ dx.
\end{align*}
Here
\begin{equation*}
\int v_5v_2+u_2u_5\ dx=-\int v_4v_3+u_3u_4\ dx=-\frac{1}{2}\int(v_3^2+u_3^2)_1\ dx=0,
\end{equation*}
\begin{equation*}
\int u_5v_2+u_2v_5\ dx=\int u_5v_2-u_5v_2\ dx=0,
\end{equation*}
\begin{align*}
\int &(uu_1)_2v_2+u_2(vv_1)_2\ dx\\
&=\int 3u_1u_2v_2+uv_2u_3+vu_2v_3+3v_1v_2u_2\ dx,
\end{align*}
\begin{align*}
\int &(vv_1)_2v_2+u_2(uv)_3\ dx\\
&=\int 3v_2^2v_1+v_3v_2v+u_3u_2v+3u_2^2v_1+3u_2v_2u_1+v_3u_2u\ dx\\
&=\int 3v_2^2v_1+\frac{1}{2}(v_2^2)_1v+\frac{1}{2}(u_2^2)_1v+3u_2^2v_1+3u_2v_2u_1+v_3u_2u\ dx\\
&=\int \frac{5}{2}(u_2^2+v_2^2)v_1+3u_2v_2u_1+v_3u_2u\ dx\\
&=\int \frac{5}{2}(u_2^2+v_2^2)v_1+3u_2v_2u_1-v_2u_3u-v_2u_2u_1\ dx\\
&=\int \frac{5}{2}(u_2^2+v_2^2)v_1+2u_2v_2u_1-u_3v_2u \ dx.
\end{align*}
By symmetry, we also have
\begin{equation*}
\int (uu_1)_2u_2+v_2(uv)_3\ dx=\int \frac{5}{2}(u_2^2+v_2^2)u_1+2u_2v_2v_1-v_3u_2v\ dx.
\end{equation*}
This proves \eqref{53}.
\medskip

Henceforth in all computations we integrate by parts and we apply Lemma \ref{l25} several times.

\emph{Proof of \eqref{54}.} We have
\begin{align*}
\left( \int u_1^2u\ dx\right)'
&=\int 2u_1u_1'u+u_1^2u'\ dx\\
&=\int -u'(2u_2u+u_1^2)\ dx\\
&=\int (2u_2u+u_1^2)(u_1u+u_3+ku+a_1v_1v+a_2(uv)_1+a_3v_3)\ dx\\
&=k\int 2u_2u^2+u_1^2u\ dx\\
&\qquad\qquad +\int u_1u(2u_2u+u_1^2)\ dx\\
&\qquad\qquad +\int u_3(2u_2u+u_1^2)\ dx\\
&\qquad\qquad +a_1\int v_1v(2u_2u+u_1^2)\ dx\\
&\qquad\qquad +a_2\int (uv)_1(2u_2u+u_1^2)\ dx\\
&\qquad\qquad +a_3\int v_3(2u_2u+u_1^2)\ dx.
\end{align*}
Here all integrals are equivalent to zero by Lemma \ref{l25}, except those containing $u_3$ or $v_3$. Since
\begin{equation*}
\int u_3(2u_2u+u_1^2)\ dx=\int (u_2^2)_1u+u_3u_1^2\ dx=-\int u_2^2u_1+2u_2^2u_1\ dx=-3\int u_2^2u_1\ dx
\end{equation*}
and
\begin{align*}
\int v_3(2u_2u+u_1^2)\ dx
&=2\int v_3u_2u-v_2u_2u_1\ dx\\
&=2\int -v_2u_3u-v_2u_2u_1-v_2u_2u_1\ dx\\
&=-2\int u_3v_2u+2u_2v_2u_1\ dx,
\end{align*}
we conclude that
\begin{equation*}
\left( \int u_1^2u\ dx\right)'\approx -3\int u_2^2u_1\ dx-2a_3\int u_3v_2u+2u_2v_2u_1\ dx.
\end{equation*}

Adding this to the analogous relationship for $\int v_1^2v\ dx$ we get \eqref{54}.
\medskip

\emph{Proof of \eqref{55} and \eqref{56}.} We have
\begin{align*}
\left( \int u_1v_1v\ dx\right)'
&=\int u_1'v_1v+u_1v_1'v+u_1v_1v'\ dx\\
&=\int -u'(v_2v+v_1^2)-v'u_2v\ dx\\
&=\int (v_2v+v_1^2)(u_1u+u_3+ku+a_1v_1v+a_2(uv)_1+a_3v_3)\ dx\\
&\qquad\qquad +\int u_2v(v_1v+v_3+kv+a_2u_1u+a_1(uv)_1+a_3u_3)\ dx\\
&\approx \int v_2vu_3+v_1^2u_3+u_2vv_3\ dx+a_3\int (v_2v+v_1^2)v_3+u_2vu_3 \ dx \\
&=\int (u_2v_2)_1v-u_2(v_1^2)_1\ dx+a_3\int (v_2v+v_1^2)v_3+u_2vu_3 \ dx \\
&= -3\int u_2v_2v_1\ dx+a_3\int (v_2v+v_1^2)v_3+u_2vu_3 \ dx.
\end{align*}
Since
\begin{align*}
\int (v_2v+v_1^2)v_3+u_2vu_3 \ dx
&=\int \frac{1}{2}(v_2^2)_1v-2v_2^2v_1+\frac{1}{2}v(u_2^2)_1\ dx\\
&=\int -\frac{1}{2}v_2^2v_1-2v_2^2v_1-\frac{1}{2}u_2^2v_1\ dx\\
&=\int -\frac{5}{2}v_2^2v_1-\frac{1}{2}u_2^2v_1\ dx,
\end{align*}
it follows that
\begin{equation*}
\left( \int 2u_1v_1v\ dx\right)'\approx -6\int u_2v_2v_1\ dx-a_3\int (5v_2^2+u_2^2)v_1\ dx,
\end{equation*}
and then by symmetry
\begin{equation*}
\left( \int 2u_1v_1u\ dx\right)'\approx -6\int u_2v_2u_1\ dx-a_3\int (5u_2^2+v_2^2)u_1\ dx.
\end{equation*}

Next we have
\begin{align*}
\left( \int u_1^2v\ dx\right)'
&=\int 2u_1u_1'v+u_1^2v'\ dx\\
&=\int -(2u_2v+2u_1v_1)u'+u_1^2v'\ dx\\
&=\int (2u_2v+2u_1v_1)(u_1u+u_3+ku+a_1v_1v+a_2(uv)_1+a_3v_3)\ dx\\
&\qquad\qquad +\int -u_1^2(v_1v+v_3+kv+a_2u_1u+a_1(uv)_1+a_3u_3)\ dx\\
&\approx \int 2u_3u_2v+2u_1v_1u_3-u_1^2v_3\ dx\\
&\qquad\qquad +a_3\int (2u_2v+2u_1v_1)v_3-u_1^2u_3 \ dx\\
&=\int -u_2^2v_1-2u_2(u_1v_1)_1+2u_1u_2v_2\ dx\\
&\qquad\qquad +a_3\int (2u_2v+2u_1v_1)v_3-u_1^2u_3 \ dx\\
&=-3\int u_2^2v_1\ dx+a_3\int (2u_2v+2u_1v_1)v_3-u_1^2u_3 \ dx.
\end{align*}
Since
\begin{align*}
\int (2u_2v+2u_1v_1)v_3-u_1^2u_3 \ dx
&=\int -2v_2(u_3v+2u_2v_1+u_1v_2)+2u_2^2u_1\ dx\\
&=\int -2u_3v_2v-4u_2v_2v_1-2v_2^2u_1+2u_2^2u_1\ dx\\
&=2\int v_3u_2v-u_2v_2v_1+(u_2^2-v_2^2)u_1\ dx,
\end{align*}
it follows that
\begin{equation*}
\left( \int u_1^2v\ dx\right)'=-3\int u_2^2v_1\ dx+2a_3\int v_3u_2v-u_2v_2v_1+(u_2^2-v_2^2)u_1\ dx,
\end{equation*}
and then by symmetry
\begin{equation*}
\left( \int v_1^2u\ dx\right)'=-3\int v_2^2u_1\ dx+2a_3\int u_3v_2u-u_2v_2u_1+(v_2^2-u_2^2)v_1\ dx.
\end{equation*}
Combining the four relations we get \eqref{55} and \eqref{56}.
\end{proof}

\begin{proof}[Proof of the proposition for $n=2$]
We consider the functions $f, g, h$ of Lemma \ref{l51}.
If $a_3=0$ or if $a_1=a_2=1$, then $h=0$. If $\abs{a_3}<1$, then
\begin{equation*}
\int u_n^2+v_n^2\ dx\le \frac{1}{1-\abs{a_3}}\int u_n^2+v_n^2+2a_3u_nv_n\ dx.
\end{equation*}
Since by Lemma \ref{l25} and the induction hypothesis $f$ and $g$ satisfy the assumptions of Lemma \ref{l23}, we may conclude as in case $n=1$ above.
\end{proof}

\section{Proof of the proposition for $n\ge 3$}\label{s6}

We proceed by induction on $n$, so we assume that the proposition holds for smaller values of $n$.

By Lemma \ref{l31} we have
\begin{align}\label{61}
\left( \int u_n^2+v_n^2\ dx\right) '
&=-2k\int u_n^2+v_n^2\ dx\\
&\qquad\qquad -2\int u_n(u_1u)_n+v_n(v_1v)_n\ dx\notag \\
&\qquad\qquad -2a_1\int u_n(vv_1)_n+v_n(uv)_{n+1}\ dx\notag \\
&\qquad\qquad -2a_2\int v_n(uu_1)_n+u_n(uv)_{n+1}\ dx.\notag
\end{align}
If we differentiate the products in the last three integrals by using Leibniz's rule and the binomial formula, we obtain a sum of three-term products. Using the inequality $n\ge 3$, it follows from Lemma \ref{l25} that all terms are equivalent to zero, except those containing the factor $u_{n+1}$ or $v_{n+1}$.

Indeed, the orders of differentiation of the three factors are $n$, $j$ and $n+1-j$ with $1\le j\le n$. Since the sum $2n+1$ of the differentiations satisfies the inequality $2n+1<2n+(n-1)$, we have
\begin{equation*}
2(\alpha_n+\beta_n)+(\alpha_{n_1}+\beta_{n-1})\le 4,
\end{equation*}
and Lemma \ref{l25} applies.

Using again that $1\le n-2$, it follows  that
\begin{align*}
\int u_n(u_1u)_n+v_n(v_1v)_n\ dx
&\approx \int u_nu_{n+1}u+v_nv_{n+1}v\ dx\\
&= \frac{1}{2}\int (u_n^2)_1u+(v_n^2)_1v\ dx\\
&=-\frac{1}{2}\int u_n^2u_1+v_n^2v_1\ dx\\
&\approx 0,
\end{align*}

\begin{align*}
\int u_n(vv_1)_n+v_n(uv)_{n+1}\ dx
&\approx \int u_nvv_{n+1}+v_nu_{n+1}v+v_nuv_{n+1}\ dx\\
&=\int u_nvv_{n+1}-u_n(v_nv)_1+\frac{1}{2}u(v_n^2)_1\ dx\\
&=\int -u_nv_nv_1-\frac{1}{2}u_1v_n^2\ dx\\
&\approx 0,
\end{align*}
and by symmetry
\begin{equation*}
\int v_n(uu_1)_n+u_n(uv)_{n+1}\ dx\approx 0.
\end{equation*}
Using these relations we infer from \eqref{61} that
\begin{equation*}
\left( \int u_n^2+v_n^2\ dx\right) '\approx -2k\int u_n^2+v_n^2\ dx,
\end{equation*}
and we conclude as usual.
\medskip

\emph{Acknowledgement.} 
The first and third authors were supported by CNPq (Brazil). The second author was supported by PROEX-Capes (Brazil).
Part of this work was done  during the visit of the first author to the University of Strasbourg in November 2012, and the visit of the second author to the Federal University of Rio de Janeiro (UFRJ) in February--March 2012. The authors thank the host institutions for their warm hospitality.

\end{document}